\documentclass[10pt]{amsart}
\usepackage[utf8]{inputenc}
\usepackage{fontenc}
\usepackage{amsfonts}
\usepackage{amssymb}
\usepackage{amsmath}
\usepackage{amsthm}\usepackage{mathtools}
\usepackage{enumerate}
\usepackage{hyperref}\usepackage{enumitem}
\usepackage{mathrsfs}
\usepackage{tikz}
\usepackage{marginnote}
\numberwithin{equation}{section}

\newcommand{\e}{\varepsilon}

\newcommand{\cE}{\mathcal{E}}
\newcommand{\cQ}{\mathcal{Q}}

\theoremstyle{Case1}

\theoremstyle{Case2}

\newcommand{\NN}{\mathbb{N}}

\newtheorem*{rigprob*}{Rigidity Problem for Uniform Roe Algebras}
\newtheorem*{rigprobcorona*}{Rigidity Problem for Uniform Roe Coronas}

\newcommand{\cst}{\mathrm{C}^*}
\newcommand{\cstar}{$\mathrm{C}^*$}

\newcommand{\cU}{\mathcal{U}}

\newcommand{\bbN}{\mathbb{N}}
\newcommand{\bbC}{\mathbb{C}}

\newcommand{\bbR}{\mathbb{R}}

\newcommand{\cO}{\mathcal{O}}

\newcommand{\ZFC}{\mathrm{ZFC}}

\newcommand{\CH}{\mathrm{CH}}

\newtheorem{thm}{Theorem}
\newtheorem{coro}[thm]{Corollary}
\newtheorem{theorem}{Theorem}[section]
\newtheorem*{theorem*}{Theorem}
\newtheorem{proposition}[theorem]{Proposition}

\newtheorem*{proposition*}{Proposition}
\newtheorem{lemma}[theorem]{Lemma}
\newtheorem*{lemma*}{Lemma}
\newtheorem{corollary}[theorem]{Corollary}
\newtheorem*{corollary*}{Corollar}

\newtheorem*{fact*}{Fact}
\theoremstyle{definition}
\newtheorem{definition}[theorem]{Definition}
\newtheorem*{definition*}{Definition}
\newtheorem{claim}[theorem]{Claim}
\newtheorem*{claim*}{Claim}

\newtheorem*{conjecture*}{Conjecture}

\newtheorem{question}[theorem]{Question}

\theoremstyle{remark}

\newtheorem*{example*}{Example}

\newtheorem*{remark*}{Remark}

\newtheorem*{note*}{Note}
\newtheorem*{question*}{Question}

\newcommand{\norm}[1]{\left\lVert #1 \right\rVert}

\DeclareMathOperator{\jsp}{jsp}
\DeclareMathOperator{\dist}{dist}

\DeclareMathOperator{\supp}{supp}
\DeclareMathOperator{\id}{id}

\newcommand{\cM}{\mathcal M} 
\newcommand{\calK}{\mathcal K} 
\newcommand{\bbD}{\mathbb D} 
\newcommand{\bt}{\mathbf t}

\newcounter{my_enumerate_counter}
\newcommand{\pushcounter}{\setcounter{my_enumerate_counter}{\value{enumi}}}
\newcommand{\popcounter}{\setcounter{enumi}{\value{my_enumerate_counter}}}

\usepackage{enumitem}

\DeclareMathOperator{\End}{End}
\begin{document}

\title{Obstructions to countable saturation in corona algebras}%

\author[I. Farah]{Ilijas Farah}
\address[Ilijas Farah]{Department of Mathematics and Statistics\\
York University\\
4700 Keele Street\\
North York, Ontario\\ Canada, M3J
1P3 and Matemati\v{c}ki Institut SANU, Kneza Mihaila 36, Belgrade 11001, Serbia
(\emph{ORCiD:} \href{https://orcid.org/0000-0001-7703-6931}{0000-0001-7703-6931}).}
\email{ifarah@yorku.ca}
\urladdr{https://ifarah.mathstats.yorku.ca}

\author[A. Vignati]{Alessandro Vignati}
\address[Alessandro Vignati]{Institut de Math\'ematiques de Jussieu (IMJ-PRG)\\
Universit\'e Paris Cit\'e\\
B\^atiment Sophie Germain\\
8 Place Aur\'elie Nemours \\ 75013 Paris, France (\emph{ORCiD:} \href{https://orcid.org/0000-0002-8675-3657}{0000-0002-8675-3657}).}
\email{vignati@imj-prg.fr}
\urladdr{https://www.automorph.net/avignati}

\thanks{}
\date{\today}%
\maketitle
\begin{abstract}
We study the extent of countable saturation for coronas of abelian \cstar-algebras. In particular, we show that the corona algebra of $C_0(\bbR^n)$ is countably saturated if and only if $n=1$.
\end{abstract}
\section{Introduction}

Saturated models are an essential tool in model theory. For example, sufficiently saturated `monster' models are universal and homogeneous, and as such provide the ambient for stability-theoretic considerations. Our motivation for studying saturation of \cstar-algebras (viewed as metric structures, as in \cite{BYBHU}, \cite{FaHaSh:Model2}, or \cite{Muenster}) derives from the following two closely related facts. First, an infinite (noncompact) saturated model of cardinality (or, in the case of continuous model theory, density character $\kappa$) has $2^\kappa$ automorphisms. Second, saturated models of the same cardinality (or of the same density character) are isomorphic if and only if they are elementarily equivalent. 

The corona algebra (see \S\ref{S.Definitions} below for the definition) of any $\sigma$-unital \cstar-algebra satisfies a weakening of countable saturation, called \emph{countable degree-1 saturation} (\cite{FaHa:Countable}, also \cite[\S 15]{Fa:STCstar}). This property implies virtually all saturation-like properties of coronas used in the literature (such as the properties of ultrapowers isolated by Kirchberg in \cite{Kirc:Central}, and many more, see \cite[\S 15.3 and \S 15.4]{Fa:STCstar}, or \cite[\S5.2]{farah2022corona}). However, it is not clear whether countable degree-1 saturation suffices for construction of isomorphisms or automorphisms. Also, the Calkin algebra is not countably quantifier-free saturated, and not even countably homogeneous (\cite[Example 16.1.1]{Fa:STCstar}, \cite[\S 4]{FaHa:Countable}, and \cite{farah2015calkin}). This makes constructing outer automorphisms of the Calkin algebra rather complicated (but not impossible, granted the Continuum Hypothesis ($\CH$)---see \cite{PhWe:Calkin} and \cite[\S 17.1]{Fa:STCstar}). We do not know whether there exists a simple nonunital separable \cstar-algebra whose corona is countably saturated or at least countably homogeneous (see Question~\ref{Q.simple}). 

In \cite{CoFa:Automorphisms} it was conjectured that if $A$ is a separable, nonunital \cstar-algebra then $\CH$ implies that the corona $\cQ(A)$  has $2^{2^{\aleph_0}}$ automorphisms. (The other conjecture of this paper, dual to this one, according to which forcing axioms imply that $\cQ(A)$ has only (appropriately defined) trivial automorphisms, has been confirmed in \cite{vignati2018rigidity}.) 

If $A$ is a separable nonunital \cstar-algebra whose corona is countably saturated, then $A$ clearly satisfies the first conjecture, but very few coronas of separable \cstar-algebras are known to be countably saturated. By \cite[Theorem~1.5]{FaSh:Rigidity}, if $A_n$ is a sequence of unital \cstar-algebras then the corona algebra of $\bigoplus_n A_n$ is countably saturated. This, together with a generalization of the Feferman--Vaught theorem, was used in \cite{ghasemi2016reduced} and \cite{Gha:FDD} (see also \cite{mckenney2018forcing} and \cite{vignati2018rigidity}) to find sequences of pairwise nonisomorphic simple unital separable \cstar-algebras $A_n$ and $B_n$ such that the isomorphism of coronas of $\bigoplus A_n$ and $\bigoplus B_n$ is independent from $\ZFC$; see \cite{farah2022corona} for the current state of the art on the conjectures of \cite{CoFa:Automorphisms}. Under $\CH$, a countably saturated structure of density character $2^{\aleph_0}$ is isomorphic to an ultrapower. Thus countable saturation of reduced products was used in \cite{farah2019between} to prove that the quotient map from $\ell_\infty(A)/c_0(A)$ to a (norm) ultrapower~$A_\cU$ associated with a nonprincipal ultrafilter $\cU$ on $\bbN$ has a right inverse if $\CH$ holds. This implies (in $\ZFC$) a transfer theorem for realization of morphisms with respect to classification functors (\cite[Theorem~A]{farah2019between}).

Our main interest in the present paper is in the saturation of coronas of abelian \cstar-algebras, a study initiated in \cite{eagle2015saturation}. If $X$ is a locally compact\footnote{All spaces in this note are Hausdorff.}, noncompact, space that is an increasing union of compact subsets $K_n$ such that $\sup_n|\partial K_n|<\infty$, where $\partial$ denotes the topological boundary, then the corona algebra of $C_0(X)$ is countably saturated (\cite[Theorem~2.5]{FaSh:Rigidity}). This covers all known examples of countably saturated coronas of abelian \cstar-algebras. The coronas of the \cstar-algebras of the form $A=C_0([0,\infty),B)$ for a separable, unital, and infinite-dimensional \cstar-algebra $B$ need not be countably quantifier-free saturated. (See Question~\ref{Q.Ainfty}, also \cite[Exercise~16.8.36]{Fa:STCstar}; the point of the latter is that the Murray-von Neumann semigroup of $\cQ(A)$ is isomorphic to that of $B$, and it can therefore be countably infinite.) 

Since the corona algebra of $C_0(\bbR)$ is countably saturated, $\CH$ implies that it has~$2^{2^{\aleph_0}}$ automorphisms. By an intricate construction in \cite{vignati2017nontrivial} it was proved that under $\CH$ the corona algebra of $C_0(\mathbb R^n)$, for $n\geq 2$ (as well as many other coronas with similar properties) also has $2^{2^{\aleph_0}}$ automorphisms. The question whether these coronas are countably saturated remained open until now. 

\begin{thm}\label{T0}
Suppose that $D$ is an infinite discrete space and that $Y$ is locally compact, $\sigma$-compact, noncompact, and connected. If $D\times Y$ embeds as a closed subspace in a locally compact space $X$, then the corona algebra of $C_0(X)$ is not countably quantifier-free saturated.
\end{thm}

\begin{coro}\label{C00}
Let $X$ be a locally compact space in which $\mathbb R^2$ embeds as a closed subspace. Then $C_0(X)$ is not countably saturated. In particular, for every $n\geq 1$ the following are equivalent. \begin{enumerate}
\item The corona algebra of $C_0(\bbR^n)$ is countably saturated. 
\item The corona algebra of $C_0(\bbR^n)$ is countably quantifier-free saturated. 
 \item $n=1$. 
 \end{enumerate}
\end{coro}

The question whether, for \cstar-algebras, countable saturation is equivalent to countable quantifier-free saturation was raised in \cite[Question~6.5]{eagle2015saturation}. While Corollary~\ref{C00} gives some support for the positive answer, by Proposition~\ref{P.QFsaturated} below, the answer is in general negative. 

%\end{corollary}
%\begin{coro}\label{C0}
%Let $X$ be a locally compact space in which $\mathbb R^2$ embeds as a closed subspace. Then $C_0(X)$ is not countably saturated.
%\end{coro}

The paper is organised as follows. In \S\ref{S.elementary} we give a rather elementary topological proof of Theorem~\ref{T0}, and in \S\ref{S.limiting} we state some corollaries and limiting examples. In \S\ref{S.homotopy} we introduce the notion of \emph{definable homotopy} and give another proof of the same theorem. This, a bit more involved proof, is more likely to be adaptable to a more general case. In \S\ref{S.redprod} we show that the corona algebra of a direct sum $\bigoplus_n A_n$ of a sequence of nonunital \cstar-algebras is not necessarily countably saturated, contrasting the result of \cite[Theorem~1.5]{FaSh:Rigidity} mentioned earlier. In \S\ref{S.Nonabel} we answer a question of Eagle and the second author (\cite[Question~6.5]{eagle2015saturation}), by finding a family of \cstar-algebras that are countably quantifier-free saturated but not countably saturated. 

We assume familiarity with logic of metric structures, as applied to \cstar-algebras (\cite{Muenster}; the relevant definitions of a condition, type, and saturation can be found in \cite[\S 16]{Fa:STCstar}).

\subsection{Multipliers and coronas} \label{S.Definitions} 
If $A$ is a \cstar-algebra, its multiplier algebra $\cM(A)$ is the unital \cstar-algebra such that when a unital \cstar-algebra $B$ contains $A$ as an essential ideal\footnote{An ideal $A\subseteq B$ is \emph{essential} if $A^\perp=\{b\in B\mid bA=Ab=0\}$ is $0$}, then the identity map on $A$ extends uniquely to a $^*$-homomorphism from $B$ to $\cM(A)$ (see \cite[II.7.3.1]{Black:Operator} or \cite[Corollary~13.2.3]{Fa:STCstar}). This is a noncommutative generalisation of the \v{C}ech--Stone compactification; in fact, if $A=C_0(X)$ for a locally compact space $X$ then $\cM(A)\cong C(\beta X)$. The corona algebra of $A$ is the quotient $\cQ(A)=\cM(A)/A$, and $\pi$ denotes the canonical quotient map. 

\subsection*{Acknowledgments} The results of \S\ref{S.homotopy} were obtained during IF's visit to Universit\'e Paris Diderot (currently Universit\'e Paris Cit\'e). IF is partially supported by NSERC. AV is supported by an Emergence en Recherche grant from IdeX-Universit\'e Paris Cit\'e and partially by the ANR project AGRUME (ANR-17-CE40-0026).

We are grateful to Logan Hoehn (see the last paragraph of \S\ref{S.limiting}) and to S\o ren Eilers for turning our attention to \cite{sorensen2012characterization} and \cite{enders2019almost}. 

\section{An elementary proof}\label{S.elementary}

The proof of Theorem~\ref{T1} is based on an idea of the proof of \cite[Theorem~2.2]{FaSh:Rigidity}. 

\begin{theorem}\label{T1}
Suppose that $X$ is a locally compact and $\sigma$-compact space such that there are infinitely many closed disjoint connected subsets $Y_n$ of~$X$ with none of them compact and such that each compact subset of $X$ intersects only finitely many $Y_n$'s. Then the corona algebra of $C_0(X)$ is not countably quantifier-free saturated.
\end{theorem} 

\begin{proof} Since $X$ has a noncompact closed subset, it is not compact either. 
Let $C_b(X)$ denote the \cstar-algebra of bounded continuous complex-valued functions on $X$. We can identify $\cQ(C_0(X))$ with $C_b(X)/C_0(X)$. Let $\pi\colon C_b(X)\to \cQ(C_0(X))$ denote the quotient map. For brevity, for $c\in C_b(X)$ we write $\dot c=\pi(c)$.

Write $X$ as an increasing union of compact subspaces, $X_n=\bigcup_n K_n$. First find open sets $U_n$, $V_n$ and $W_n$ such that for all $n$ the following conditions hold. 
\begin{enumerate}
\item\label{IT1.1} $Y_n\subseteq U_n\subseteq\overline U_n\subseteq V_n\subseteq\overline V_n\subseteq W_n$,
\item \label{IT1.2} all the sets $\overline W_n$ are disjoint, 
\item \label{IT1.2.5} $K_m\cap W_n=\emptyset$ if $K_m\cap Y_n=\emptyset$ for all $m$, and
\item \label{IT1.3} the sets $U_n$ are connected.
%\item \label{IT1.4} $K_m\cap V_n=\emptyset$ if $m\leq n$, and 
\end{enumerate}
For this, first use Uryshon's Lemma to obtain $U_n'$, $V_n'$ and $W_n$ satisfying \eqref{IT1.1}, \eqref{IT1.2}, and \eqref{IT1.2.5}. Then let $U_n$ be the connected component of $U_n'$ that contains~$Y_n$. 

We will need positive contractions $a_n\in C_b(X)$ satisfying the following. 
\begin{enumerate}[resume]
\item $a_n(t)=1$ for $t\in Y_n$,
\item $a_n(t)\geq \frac 23$ for $t\in \overline{U_n}$, 
\item $a_n(t)\leq \frac 23$ for $t\in X\setminus U_n$, and 
\item $\supp(a_n)\subseteq V_n$. 
\end{enumerate}
In order to find $a_n$, using the Tietze extension theorem first fix a continuous $a_n^0\colon X\to [0,\frac 23]$ such that $\supp(a_n^0)\subseteq V_n$ and $a_n^0(t)=\frac 23$ for $t\in \overline{U_n}$. Then find continuous $a_n^1\colon X\to [0,\frac 13]$ such that $\supp(a_n^1)\subseteq U_n$ and $a_n^1(t)=\frac 13$ for all $t\in Y_n$ and let $a_n=a_n^0+a_n^1$. 

Since on every compact subset of $X$ at most finitely many of the summands are nonempty and all the $a_i$'s are orthogonal, $a=\sum_j a_j$ is in $C_b(X)$. Let $b_n\in C_b(X)$ be a positive contraction such that $b_n(t)=1$ for all $t\in \overline V_n$ and $\supp(b_n)\subseteq W_n$. Notice that $a_n b_n=a_n$ and $a_m b_n=0$ for all $n$ and all $m\neq n$. In particular $b_na=b_n^2a=a_n$ for all $n$.

Consider the type $\bt(x)$ over $\cQ(X)$ with the conditions 
\begin{enumerate}
\item [] $\|x\dot a\|=1$, $x\geq 0$, $(x-x^2)\dot a=0$, and $x\dot a_n=0$ for all $n$. 
\end{enumerate}
It will suffice to show that $\bt(x)$ is consistent, but not realized in $\cQ(X)$. 

If $\bt_0$ is a finite subset of $\bt$, fix $n$ such that $\dot a_n$ does not appear in $\bt_0$. Then~$\dot b_n$ realizes $\bt_0$ (it realizes every condition of $\bt$ except $x \dot a_n=0$), therefore $\bt$ is consistent. 

It remains to prove that $\bt$ is not realized in $\cQ(X)$. Assume otherwise, fix $c\in C_b(X)$ such that $\dot c$ realizes $\bt$. Since $(\dot c-\dot c^2)\dot a=0$, we can fix $m$ such that $((c-c^2)a)(t)<\frac 16$ for all $t\in X\setminus K_m$. Fix for a moment $t\in \overline {\bigcup_n U_n}\setminus K_m$. Then $\frac 16>((c-c^2)a)(t)\geq \frac 23 (c-c^2)(t)$. Therefore $(c-c^2)(t)<1/4$ and so
\begin{enumerate}
\item [] $c(t)\neq 1/2$ for all $t\in \overline {\bigcup_n U_n}\setminus K_m$. 
\end{enumerate}
 In addition, by the assumption that each $K_m$ intersects only finitely many~$Y_n$ nontrivially and \eqref{IT1.2.5}, the set 
\[
F=\{n\mid U_n\cap K_m\neq \emptyset\}
\]
is finite. Since $\dot c \dot a_n=0$ for all $n\in F$, there is $m'\geq m$ such that for all $n\in F$ we have $(c a_n)(t)<1/2$ for all $t\in X\setminus K_{m'}$. Since $\|\dot c\dot a\|=1$, we can find $s_0\in X\setminus K_{m'}$ such that $\frac{2}{3}<(ca)(s_0)$. This implies $\frac 23 < c(s_0)$, and since $\{s\mid ca(s)> \frac{2}{3}\}\subseteq\{s\mid a(s)> \frac{2}{3}\}\subseteq\bigcup U_n$ we have $s_0\in \bigcup_n U_n$. Fix $n$ such that $s_0\in U_n$. Then $n\notin F$ and therefore $U_n\cap K_{m}=\emptyset$. 

Since $\dot c\dot a_n=0$, some $s_1\in Y_n$ satisfies $\frac 12>c(s_1)$; if not, $ca_n(s)\geq \frac{1}{2}$ for all $s\in Y_n$, and therefore, as $Y_n$ is not compact, $\norm{ca_n}\geq\frac{1}{2}$. But $Y_n\subseteq U_n$, hence $c$ attains values both above and below~$\frac 12$ on the connected space $U_n=U_n\setminus K_m$. Hence some $t\in U_n$ satisfies $c(t)=\frac 12$; contradiction. 
\end{proof}

\section{Corollaries and a limiting example} \label{S.limiting}

Theorem~\ref{T1} resolves the question of countable saturation for many, but not all, coronas of abelian \cstar-algebras. 

\begin{proof}[Proof of Theorem~\ref{T0}]
Suppose that $D=\{d_n\}$ is an infinite discrete set, that~$Y$ is locally compact, $\sigma$-compact, noncompact, and connected, and that $D\times Y$ embeds as a closed subspace in a locally compact space $X$. By discarding the extras, we can assume $D$ is countable and enumerate it as $d_n$, for $n\in \bbN$. Let $Y_n=\{d_n\}\times Y$. To check the assumptions of Theorem~\ref{T1}, suppose $K$ is a compact subset of $X$. Then $K\cap (D\times Y)$ is compact, and therefore $K\cap Y_n$ is nonempty for at most finitely many $n$. By Theorem~\ref{T1}, the corona algebra of $C_0(X)$ is not countably quantifier-free saturated. 
\end{proof}

\begin{proof}[Proof of Corollary~\ref{C00}]
The corona algebra of $C_0(\bbR)$ is countably saturated by \cite[Theorem~2.5]{FaSh:Rigidity}. 

It remains to prove that the corona algebra of $C_0(\bbR^n)$ for $n\geq 2$ is not countably quantifier-free saturated. Since $\bbN\times \bbR$ is a closed subset of $\bbR^2$ (and in turn of $\bbR^n$ for all $n\geq 2$), if a locally compact space $X$ has $\bbR^2$ as a closed subspace then Theorem~\ref{T0} implies that the corona of $C_0(X)$ is not countably quantifier-free saturated.
\end{proof}

%\begin{corollary}\label{C0}
%Suppose that $X$ is an infinite discrete set, and that $Y$ is locally compact, $\sigma$-compact, noncompact, and connected. Suppose that $X\times Y$ embeds as a closed subspace in a locally compact space $Z$. Then the corona algebra of $C_0(Z)$ is not countably quantifier-free saturated. \qed 
%\end{corollary}

The proof of Corollary~\ref{C00} may seem to suggest that if $X$ and $X'$ are locally compact, noncompact spaces such that $X$ is a closed subspace of $X'$ and the corona algebra $C_0(X)$ is not countably quantifier-free saturated, then the corona algebra of $C_0(X')$ is not countably quantifier-free saturated. It is however not clear whether this is true in general. 

If $Y$ is sufficiently connected at infinity, the space $D$ in Theorem~\ref{T0} can be assumed to be locally compact instead of discrete.

\begin{corollary}\label{C1}
Suppose that $D$ is an infinite locally compact space and that $Y$ is locally compact and noncompact. Suppose furthermore that $Y$ can be written as an increasing union of compact sets $K_n$ such that $\overline{Y\setminus K_n}$ is connected for all $n$, and that $D\times Y$ embeds as a closed subspace in a locally compact space $X$. Then the corona algebra of $C_0(X)$ is not countably quantifier-free saturated. 
\end{corollary}

\begin{proof} 
Choose an infinite discrete subset $\{d_n\mid n\in \bbN\}$ of $D$. Then the subspaces $Y_n=\{d_n\}\times (\overline {Y\setminus K_n})$ of $X$ are closed, connected, and disjoint, and $\bigcup_n Y_n$ is a closed subspace of $X$. If $K\subset X$ is compact, then so is $K\cap (D\times Y)$ and therefore the latter set is included in some $D\times K_n$. Hence $K$ intersects at most finitely many $Y_n$ nontrivially. By Theorem~\ref{T1}, the corona algebra of $C_0(X)$ is not countably quantifier-free saturated. 
\end{proof}

The conclusion of Corollary~\ref{C1} may fail if $Y$ is not assumed to be $\sigma$-compact. This is Corollary~\ref{Ex.notsigmacompact}, an easy consequence of the following. 

\begin{proposition}\label{P.any} There is a locally compact connected space $Y$ such that $B\cong \mathcal Q(C_0(Y,B))$ for every unital \cstar-algebra $B$. 
\end{proposition}

\begin{proof} Ignoring the connectedness requirement for a moment, consider $\aleph_1$ with the order topology. Then the multiplier algebra of $A=C_0(\aleph_1, B)$ can be identified with $C_b(\aleph_1, B)$ (the \cstar-algebra of bounded continuous functions from $\aleph_1$ into $B$); this is a very easy case of \cite[Theorem~3.3]{akemann1973multipliers}. Every continuous function $f$ from $\aleph_1$ into a metric space is eventually constant. We include a proof of this well-known fact. For $\alpha<\aleph_1$ let $\varepsilon(\alpha)$ denote the diameter of $\{f(\beta)\mid \alpha\leq \beta< \aleph_1\}$. Since $\alpha\mapsto \varepsilon(\alpha)$ is a nonincreasing function, if $f$ is not eventually constant, then there is $\delta>0$ such that $\varepsilon(\alpha)\geq \delta$ for all $\alpha$. We can now recursively find two increasing sequences of ordinals $\alpha_n$ and $\beta_n$ such that $\sup_n \alpha_n=\sup_n \beta_n$ and $d(f(\alpha_n),f(\beta_n))>\delta/2$. Thus $f$ is discontinuous at $\sup_n \alpha_n$; contradiction. 

Therefore the diagonal map of $B$ into $C_b(\aleph_1,B)/C_0(\aleph_1, B)$ is surjective. Since it is clearly injective, $B$ is isomorphic to the corona. 

In order to prove that $Y$ can be chosen to be connected, let $Y$ be the so-called long line, defined as follows. To the space $\aleph_1$, between every ordinal $\alpha$ and its successor $\alpha+1$, attach a copy of $[0,1]$ (with the endpoints identified with $\alpha$ and $\alpha+1$). Every $f\colon Y\to B$ is eventually constant, via a proof analogous to the corresponding fact for $f\colon \aleph_1\to B$. As before, for $A=C_0(Y, B)$ this implies that the diagonal map of $B$ into $\cQ(A)$ is an isomorphism. 
\end{proof}

\begin{corollary}\label{Ex.notsigmacompact} There are an infinite compact $K$ and a connected, locally compact, noncompact $Y$ such that the corona algebra of $C_0(K\times Y)$ is countably saturated. 
\end{corollary}

\begin{proof} Take $Y$ to be the long line. For every unital abelian \cstar-algebra $C(K)$, the corona algebra of $C_0(Y,C(K))\cong C_0(Y\times K)$ is by Proposition~\ref{P.any}, isomorphic to $C(K)$. If $K=\beta\bbN\setminus \bbN$ then $C(K)$ is countably saturated by \cite{eagle2015saturation}, or by the fact that it is the reduced product $\ell_\infty/c_0$ and therefore countably saturated by \cite[Theorem~16.5.1]{Fa:STCstar}. 
\end{proof} 
One can ask for the space $K$ in Corollary~\ref{Ex.notsigmacompact} to be connected. For this, just consider the spectrum of any nontrivial ultrapower of $C([0,1])$, as in this case $C(K)$ is countably saturated.

An open problem due to Sakai related to Proposition~\ref{P.any} is whether every unital (abelian) \cstar-algebra is isomorphic to the corona of a \emph{simple} \cstar-algebra (see \cite[Questions 5--7]{sakai1971derivations}, also \cite[Question~3.21]{farah2022corona}). The simplicity requirement suggests that some form of Proposition~\ref{P.any} was known to Sakai, although it does not appear in \cite{sakai1971derivations} explicitly. In~\cite{sakai1971derivations} Sakai shows that if $L$ is a maximal left ideal in a II$_1$ factor then the corona of $L\cap L^*$ is one-dimensional. An extension of $\calK(\ell_2(2^{\aleph_0}))$ by $\calK(\ell_2)$ whose corona is one-dimensional was constructed in \cite{ghasemi2016extension}.

If a locally compact, noncompact, space $X$ can be written as an increasing union of compact subsets $K_n$ such that $\sup_n|\partial K_n|<\infty$ then (by \cite[Theorem~2.5]{FaSh:Rigidity}) $C_b(X)/C_0(X)$ is countably saturated. There is however an example of a locally compact subspace of $\bbR^3$ without this property that does not satisfy the assumptions of Theorem~\ref{T1} either, communicated to us by Logan Hoehn. 

\section{A homotopy proof}
\label{S.homotopy}
In this section we combine model-theoretic definability (see \cite[\S 9]{BYBHU} and \cite[\S 3]{Muenster}) with some homotopy considerations. This is rather unusual, because in general it is impossible to describe the existence of a continuous path by a first-order formula. For example, the unitary group of $C([0,1])$ is path-connected, while the unitary group of its ultrapower has $2^{\aleph_0}$ connected components. 
The proof of Corollary~\ref{C00} obtained here is a bit longer than the one in \S\ref{S.elementary}, but its ideas are likely to be more relevant to analyzing the saturation of coronas of simple \cstar-algebras. 

On the set of $k$-tuples of a \cstar-algebra $A$ we consider the max norm and write 
\begin{equation}\label{Eq.metric}
\textstyle \|\bar a-\bar b\|=\max_{j<k}\|a_j-b_j\|. 
\end{equation}
A \emph{modulus of uniform continuity} is a nondecreasing $\Delta\colon (0,1]\to (0,1]$ such that $\lim_{t\to 0^+}\Delta(t)=0$. If $A$ is a \cstar-algebra, by $Z^A(\varphi)$ we denote the zero-set of $\varphi$ as computed in $A$. Given a theory $T$, we will say that the zero-set of $\varphi(\bar x)$ is \emph{definable in $T$} if there exist $\delta>0$ and a modulus of uniform continuity $\Delta_\varphi$ such that $\varphi^A(\bar b)<\delta$ implies $\dist(\bar a, Z(\varphi))<\Delta_\varphi(\varphi^A(\bar a)) $ for every model $A$ of $T$ (\cite[\S 9]{BYBHU}). If $\cE$ is the class of all models of $T$, we say that the zero-set of $\varphi(\bar x)$ is \emph{definable relative to $\cE$}. In our applications $\cE$ will be the class of abelian \cstar-algebras. 

Two elements of $Z^A(\varphi)$ are said to be \emph{homotopic} if there is a continuous function $g\colon [0,1]\to Z^A(\varphi)$ such that $g(0)=\bar a$ and $g(1)=\bar b$. The homotopy classes in $Z^A(\varphi)$ are denoted $[\bar a]_h$, and we write 
\[
Z^A(\varphi)/\sim_h
\]
for the quotient space.

\begin{definition}\label{def:Skolem}
Let $\cE$ be an axiomatizable class of \cstar-algebras, and let $\varphi(\bar x)$ be a nonnegative formula in variables $\bar x=(x_0, \dots, x_{k-1})$ whose zero-set is a definable predicate in $\cE$, witnessed by the modulus of uniform continuity~$\Delta_\varphi$. The formula $\varphi$ is said to \emph{admit a continuous Skolem function} in $\cE$ if there exist a term $g_\varphi(\bar x)$ in the language of \cstar-algebras and $\delta>0$ such that, for every $B\in\cE$,
\begin{enumerate}
 \item[]\label{Phi.1} $\varphi^B(\bar b)<\delta$ implies $\varphi^B(g_\varphi(\bar b))=0$ and $\|\bar b-g_\varphi(\bar b)\|<\Delta_\varphi(\varphi^B(\bar b))$. 
\pushcounter
\end{enumerate}
The term $g_\varphi$ is said to be the \emph{Skolem function} for $\varphi$.
\end{definition}

Some formulas that admit continuous Skolem functions are the natural formulas whose zero-sets are the set of projections and the set of unitaries (proofs use the continuous functional calculus, see \cite[\S 3]{Muenster}), as well as the example from Lemma~\ref{L.Sn} below. 

\begin{proposition} \label{P.homotopy} 
Suppose that $\cE$ is an axiomatizable class of \cstar-algebras, and that $\varphi$ admits a continuous Skolem function in $\cE$. Then there exists $\e>0$ such that for every $B\in \cE$ and any two $\bar a$ and $\bar b$ in $Z^B(\varphi)$, $\|\bar a-\bar b\|<\e$ implies that $\bar a$ and $\bar b$ are homotopic. 
\end{proposition}

\begin{proof} 
Let $g_\varphi$ be the Skolem function for $\varphi$ as in Definition~\ref{def:Skolem}. As $\varphi$, being a formula, has its own modulus of uniform continuity (\cite[Theorem~3.5]{BYBHU}), we can fix $\e>0$ small enough so that for every structure $B$ in $\cE$ and all $\bar x$ and $\bar y$ in $B$ of the appropriate sort we have that $\|\bar x-\bar y\|<\e$ implies $|\varphi^B(\bar x)-\varphi^B(\bar y)|<\delta$. 

Fix $B\in \cE$. First, notice that by the properties of $g_\varphi$, $\varphi^B(\bar a)=0$ implies $g_\varphi(\bar a)=\bar a$ for all $\bar a$ of the proper arity. Fix $\bar a$ and $\bar b$ in $Z^B(\varphi)$ such that $\|\bar a-\bar b\|<\e$. For $t\in [0,1]$ define $c_{t,j}:=t a_j +(1-t)b_j$, for $j<k$ . Let $\bar c_t=(c_{t,0},\ldots, c_{t,k-1})$. Then $\|\bar a-\bar c_t\|<\|\bar a-\bar b\|<\e$, and therefore $\varphi^B(\bar c_t)<\delta$ for all $t$. The map $t\mapsto g_\varphi(\bar c_t)$ is a continuous path in $Z^B(\varphi)$ connecting $\bar a$ and $\bar b$. 
\end{proof}

%If $\varphi(\bar x)$ is a nonnegative formula whose zero-set is definable, we let
%\[
%A_\varphi=\cst(\bar x\mid\varphi(\bar x)=0)
%\]
%denote the universal \cstar-algebra defined by the relation $\varphi(\bar x)=0$. In case the \cstar-algebra $A_\varphi$ is abelian and unital, denote by $X_\varphi$ the compact space such that $A_\varphi\cong C(X_\varphi)$. In this case, 

For a $\bar a=(a_0,\dots, a_{k-1})$ in a unital \cstar-algebra $A$, by $\jsp(\bar a)$ we denote its \emph{joint spectrum}. This is the set of tuples $\bar \lambda$ in $\bbC^k$ such that $\bar a-\bar \lambda$ generates a proper ideal in $A^k$. If $\bar a$ is a tuple of commuting normal operators, then $\cst(\bar a, 1)\cong C(\jsp(\bar a))$. 

The following definition generalises that of universal \cstar-algebra given by generators and relations (see \cite[\S II.8.3]{Black:Operator} or \cite[\S 2.3]{Fa:STCstar}), in the abelian setting.

\begin{definition}\label{defin:controlled}
Let $k\in \bbN$ and let $X\subseteq \bbR^k$ be compact. Let $\varphi(\bar x)$ be a formula in $\bar x=(x_0,\dots, x_{k-1})$. We say that $\varphi$ is \emph{controlled by $X$} (or that~$X$ \emph{controls} $\varphi$) if for every \cstar-algebra $B$ and every $k$-tuple $\bar a$ in $B$ we have $\varphi^B(\bar a)=0$ if and only if $\bar a$ is a $k$-tuple of commuting self-adjoint operators such that $\jsp(\bar a)\subseteq X$.
\end{definition}

In the following Lemma and elsewhere, a $^*$-homomorphism $\Phi\colon A\to B$ is canonically extended to all finite powers $A^n$ and if $\bar a=(a_0 \dots, a_{n-1})$ then we write $\Phi(\bar a)$ for $(\Phi(a_0), \dots, \Phi(a_{n-1}))$. 

\begin{lemma} \label{L.Phi.2} 
If $\varphi(\bar x)$ is a formula controlled by a compact $X_\varphi$ and $\Phi\colon A\to B$ is a $^*$-homomorphism between \cstar-algebras, then $\Phi[Z^A(\varphi)]\subseteq Z^B(\varphi)$.
\end{lemma} 

\begin{proof} Fix $\bar a\in Z^A(\varphi)$. Since $\bar a$ is a tuple of commuting self-adjoints, so is $\Phi(\bar a)$. By the Spectral Mapping Theorem, for every $^*$-homomorphism $\Phi$ and every $f\in C(X_\varphi)$ we have $\Phi(f(\bar a))=f(\Phi(\bar a))$ (see e.g. \cite[II.1.5.2]{Black:Operator}). %(to see this, note that this is clear in the case when $f$ is a $^*$-polynomial and apply the Stone--Weierstrass Theorem).
In particular, since $\bar r\in \jsp(\bar a)$ if and only if $\sum_j (a_j-r_j)^2$ is not invertible, we have that $\jsp(\Phi(\bar a))\subseteq\jsp(\bar a)$, which gives the thesis.
\end{proof} 

Before we continue the general analysis, an example could be helpful. 

\begin{lemma} \label{L.Sn}
For $n\geq 1$ consider the formula 
\[
\varphi_n=\textstyle\max\left(\max_{j<n} \|x_j-x_j^*\|,%\max_{i,j<n}\norm{[x_i,x_j]},
\left\|\sum_{j<n} x_j^2-1\right\|\right).
\] 
Then the following holds. 
\begin{enumerate}
%\item \label{1.L.Sn} The zero-set of $\varphi_n$ is definable in the class of all abelian \cstar-algebras. 
\item \label{2.L.Sn} $\varphi_n$ admits a continuous Skolem function in the class of abelian \cstar-algebras and
\item\label{3.L.Sn} $\varphi_n$ is controlled by $S^{n-1}$, the $n-1$-dimensional unit sphere. 
\end{enumerate}
\end{lemma}

\begin{proof}
%As \eqref{2.L.Sn} implies \eqref{1.L.Sn}, it suffices to prove \eqref{2.L.Sn} and \eqref{3.L.Sn}.
Fix $n\geq 1$. We first prove \eqref{2.L.Sn}. Let $\e>0$ be small enough $(\e<1/(5n)$ suffices) and suppose that in some abelian \cstar-algebra $B$ we have a tuple $\bar a$ such that $\varphi_n^B(\bar a)<\e$. Set $b_j=(a_j+a_j^*)/2$ and $b=\sum b_j^2$. Since $\norm{\sum a_j}\leq 1+\e$, we have $\norm{b}\leq 1+\e$. Furthermore, since $\norm{1-b}<1$, $b$ is invertible. With $c_i=b_ib^{-1/2}$ we have $\varphi_n^B(\bar c)=0$. Also, $\|\bar a-\bar c\|\leq \|\bar a-\bar b\|+\|b\|\|1-b^{-1}\|$ can be made arbitrarily small by choosing a small $\e>0$. 

Let
$f_{j}(\bar x):=\frac 12 (x_j+x_j)^*$, and let
\[
\textstyle \tilde f_{n,j}(\bar x):=f_j(\bar x)\left(\sum_j f_j(\bar x)^2\right)^{-1}
\]
The argument above shows that the function $g_n$ given by 
\[
g_n(\bar x):= (\tilde f_{n,0}(\bar x),\ldots,\tilde f_{n,n-1}(\bar x))
\]
is a Skolem function for $\varphi_n$.

\eqref{3.L.Sn} Let $B$ be a \cstar-algebra, and let $\bar a$ be a tuple in $B$ (of the appropriate length). Then $\varphi_n^B(\bar a)=0$ asserts exactly that $\bar a$ is a tuple of commuting self-adjoints and that the joint spectrum of $\bar a$ is contained in the unit sphere~$S^{n-1}$. This is precisely stating that $\varphi_n$ is controlled by $S^{n-1}$.
\end{proof} 

The reason for restricting our attention to abelian \cstar-algebras is that, rather inconveniently for our purposes, the zero-set of $\varphi_n$ is not definable in the class of all \cstar-algebras. This is because the algebra $C(S^{n-1})$, for $n\geq 3$,  is not weakly semiprojective (see \cite{sorensen2012characterization} and the introduction to \cite{enders2019almost}). This means that there is an ultraproduct of \cstar-algebras $\prod_{\cU} A_j$ and an embedding of $C(S^{n-1})$ into it that cannot be lifted to an embedding of $C(S^{n-1})$ into any of the $A_j$.  However, the machinery of this section may be applicable to other weakly stable formulas (see~\cite{Lo:Lifting}). 

We are ready to introduce the main technical tool of this section. Let $\End^1(A)$ denote the semigroup of unital endomorphisms of a unital \cstar-algebra $A$, taken with respect to composition. Suppose that $A$ is finitely generated and fix a tuple of generators $\bar a$. Thus $\alpha\in \End^1(A)$ is uniquely determined by $\alpha(\bar a)$. Equip $\End^1(A)$ with the metric defined by (using \eqref{Eq.metric}) 
\[
d(\alpha, \beta)=\|\alpha(\bar a)-\beta(\bar a)\|
\]
This metric endows $\End^1(A)$ with a natural homotopy relation $\sim_h$. 

Suppose that a formula $\varphi$ is controlled by a compact $X_\varphi\subseteq \bbR^k$ for some~$k\geq 1$. Write, for brevity, 
\[
A_\varphi:=C(X_\varphi).
\]
Let $\bar a$ denote the canonical generators of the universal \cstar-algebra $A_\varphi$, corresponding to the free variables of $\varphi$. Suppose that $B$ is a unital \cstar-algebra such that $Z^B(\varphi)$ is nonempty. Fix $\bar b\in Z^B(\varphi)$. Since $X_\varphi$ controls $\varphi$ we have $\jsp(\bar b)\subseteq X_\varphi$. This inclusion induces a $^*$-homomorphism $\Psi_{\bar b}\colon A_\varphi\to \cst(\bar b)$ such that $\Psi_{\bar b}(\bar a)=\bar b$. For $\alpha\in \End^1(A_\varphi)$ let 
\begin{equation}\label{Eq.natural}
\alpha.\bar b=\Psi_{\bar b}(\alpha(\bar a)). 
\end{equation}
Applying Lemma~\ref{L.Phi.2} twice, we have $\varphi(\alpha.\bar b)=\varphi(\alpha(\bar a))=0$. Thus \eqref{Eq.natural} defines a continuous action of $\End^1(A_\varphi)$ on $Z^B(\varphi)$. We call this action \emph{natural}. 

An action of $\End^1(A_\varphi)$ on $Z^B(\varphi)$ is said to be \emph{continuously implemented} if there is a continuous map
\begin{equation}\label{Eq.ctns}
\Phi\colon \End^1(A_\varphi)\times Z^B(\varphi)\to Z^{B}(\varphi)
\end{equation}
such that $\alpha.\bar b=\Phi(\alpha, \bar b)$. Our interest in continuously implemented actions stems from the fact that they preserve homotopy in $Z^B(\varphi)$. 

The natural action of $\End^1(A_\varphi)$ on some $Z^B(\varphi)$ is \emph{standard} if it is continuously implemented and 
there is a surjection 
\begin{equation}\label{Eq.standard}
\Psi\colon Z^B(\varphi)/\sim_h \to Z^{A_\varphi}(\varphi)/\sim_h 
\end{equation}
which is \emph{equivariant}, i.e., $\alpha.\Psi(\bar b)=\Psi(\alpha. \bar b)$. 

If $\End^1(A)$ is infinite, then its natural action on the zero-set of $\varphi$ in the ultrapower of $A$ is typically nonstandard, by countable saturation. The contrapositive of this fact serves as the main idea of the proofs in this section. 

\begin{definition} \label{Def.varphi} Suppose that $\cE$ is an axiomatizable class of \cstar-algebras, that $\varphi(\bar x)$ is a nonnegative formula in variables $\bar x=(x_0, \dots, x_{k-1})$. 
 We say that $\varphi$ \emph{admits definable homotopy} in $\cE$ if 
\begin{itemize}
\item $\varphi$ is controlled by some compact $X_\varphi\subseteq\bbR^k$, 
\item $A_\varphi=C(X_\varphi)$ is in $\cE$, and
\item $\varphi$ admits a continuous Skolem function in $\cE$. 
%\item $\varphi$ is controlled by some compact $X_\varphi\subseteq\mathbb R^k$.
\end{itemize} 
\end{definition}

%We say that a formula $\varphi$ which is controlled by some compact $X_\varphi\subseteq\bbR^k$, and admits a continuous Skolem function for the class of abelian \cstar-algebras \emph{has zero-set admitting definable homotopy}. The thesis of Proposition~\ref{P.action} justifies this terminology. 

\begin{proposition}\label{P.action} 
Suppose that a formula $\varphi$ admits definable homotopy in an axiomatizable class of \cstar-algebras $\cE$. If $B$ is a \cstar-algebra in $\cE$, then the natural action of $\End^1(A_\varphi)$ on $Z^B(\varphi)$ is continuously implemented. 
%Then for every unital abelian \cstar-algebra $B$ there is an action of $\End^1(C(X_\varphi))$ on $Z^B(\varphi)$ which is a congruence with respect to homotopy. Therefore we have an action of $\End^1(C(X_\varphi))/\sim_h$ on $\{[\bar a]_h\mid \bar a\in Z^B(\varphi)\}$. 
\end{proposition}

\begin{proof} 
Let $\bar a=(a_0,\dots, a_{k-1})$ be the generators of $A_\varphi$, so that $X_\varphi$ is homeomorphic to $\jsp(\bar a)$ and $A_\varphi\cong C(\jsp(\bar a))$. Fix for a moment $\alpha\in \End^1(A_\varphi)$. Then $\alpha$ induces a continuous map $\bar f_\alpha\colon X_\varphi\to X_\varphi$ naturally identified with a $k$-tuple $\bar f_\alpha=(f_{\alpha, j}: j<k)$ in $C(X_\varphi)$ such that $\alpha(a_j)=f_{\alpha,j}$ for $j<k$. We will summarize this situation by 
\[
\alpha(\bar a)=\bar f_\alpha.
\]
The correspondence $\alpha\mapsto \bar f_\alpha$ is a homeomorphism between $\End^1(A_\varphi)$ and the space of continuous functions from $X_\varphi$ into $X_\varphi$. 
%We make a few observations. 
%Also note that $\alpha$ is a surjective if and only if $\cst(\bar f_\alpha)=\cst(\bar a)$.\marginpar{IF: Before we had `automorphism' which is probably false; I fixed the statement, but we don't seem to need this.} 
Lemma~\ref{L.Phi.2} implies that if $\beta\in \End^1(A_\varphi)$ then $\beta(\bar f_\alpha)$ is well-defined. Clearly $\beta(\bar f_\alpha)=\bar f_{\beta\circ \alpha}$. Therefore the natural action of $\End^1(A_\varphi)$ on $Z^B(\varphi)$ (see~\eqref{Eq.natural}) satisfies 
\[
\alpha.\bar b=\bar f_\alpha(\bar b). 
\] 
This action is clearly continuous, and it therefore preserves homotopy classes. Thus $\Phi(\alpha, \bar b)=\bar f_\alpha(\bar b)$ is the required continuous functions. 
%Now fix a homotopy between $\alpha$ and $\beta$ in $\End^1(A_\varphi)$, that is, a continuous path $(\alpha_t)_{t\in [0,1]}\subseteq\End^1(A_\varphi)$ such that $\alpha=0$ and $\alpha_1=\beta$. Then $\alpha_t(\bar a)$ is a continuous path between $\alpha(\bar a)$ and $\beta(\bar a)$ for any $\bar a\in Z^B(\varphi)$. Hence if $\alpha\sim_h\beta$ then $[\alpha.\bar a]_h=[\beta.\bar a]_h$ for all $\bar a\in Z^B(\varphi)$; this is the thesis. 
\end{proof}

For a unital \cstar-algebra $A$ we write 
\begin{equation}
\label{Eq.Ainfty}
A_\infty:=C_b([0,\infty),A)/C_0([0,\infty),A). 
\end{equation}

\begin{lemma} \label{L.standard} Suppose that a formula $\varphi$ admits definable homotopy in an axiomatizable class of \cstar-algebras $\cE$. Then \begin{enumerate}
\item \label{1.L.standard} the natural action of $\End^1(A_\varphi)$ on $Z^{A_\varphi}(\varphi)$ is standard;
\item \label{2.L.standard} if $B$ is a unital \cstar-algebra in $\cE$ and the action of $\End^1(A_\varphi)$ on~$Z^B(\varphi)$ is standard, then the natural action of $\End^1(A_\varphi)$ on $Z^{B_\infty}(\varphi)$ is standard; and 
\item \label{3.L.standard} the action of $\End^1(A_\varphi)$ on $Z^{(A_\varphi)_\infty}(\varphi)$ is standard. 
\end{enumerate}
\end{lemma}

\begin{proof} By Proposition~\ref{P.action}, each of the actions in \eqref{1.L.standard}--\eqref{3.L.standard} is continuously implemented. We fix $\Phi$ as in \eqref{Eq.ctns} for each one of them and prove that $\Psi$ as in \eqref{Eq.standard} exists.

For \eqref{1.L.standard}, take $\Psi=\id_{Z^{A_\varphi}(\varphi)/\sim_h}$. 

\eqref{2.L.standard} Fix $\Phi\colon \End^1(A_\varphi)\times Z^B(\varphi)\to Z^B(\varphi)$ and $\Psi\colon Z^B(\varphi)/\sim_h \to Z^{A_\varphi}(\varphi)/\sim_h$ as in \eqref{Eq.ctns} and \eqref{Eq.standard}. Let $\delta>0$ and a continuous Skolem function $g_\varphi$ for $\varphi$ be as in Definition~\ref{def:Skolem}. By Proposition~\ref{P.homotopy} there is $\e>0$ such that for every unital abelian \cstar-algebra $D$ and any two $\bar a$ and $\bar b$ in $Z^B(\varphi)$, $\|\bar a-\bar b\|<\e$ implies that $\bar a$ and $\bar b$ are homotopic. Fix $0<\delta_0<\delta$ such that $\varphi(\bar a)<\delta_0$ implies $\|g_\varphi(\bar a)-\bar a\|<\varepsilon/3$. 

Let $C:=B_\infty$, and denote by $\pi$ the canonical quotient map $\pi\colon C_b([0,\infty),B)\to C$. Fix $\bar b\in C_b([0,\infty),B)$ such that $\pi(\bar b)\in Z^{C}(\varphi)$. Hence for some $t_0$ and all $t\geq t_0$ we have $\varphi(\bar b(t))<\delta_0$. Let $\bar b'(t)=g_\varphi(\bar b(\min(t,t_0)))$. Then $\bar b'\in C_b([0,\infty),B)$ and $\pi(\bar b')\in Z^C(\varphi)$. Also, $\pi(\bar b')=\pi(\bar b)$, and for all $t$, $\bar b'(t)$ belongs to the homotopy class of $\bar b'(0)$; for a moment let $\Theta_0(\bar b)=\bar b'(0)$. 

\begin{claim} If $\pi(\bar b)\sim_h \pi(\bar c)$, then $\Theta_0(\bar b)\sim_h \Theta_0(\bar c)$.
\end{claim}

\begin{proof} If $t$ is large enough to have $\max(\varphi(b(t)), \varphi(c(t))<\delta_0$ and $|b(t)-c(t)|<\varepsilon/3$, then (by the choice of $\delta_0$) we have $\|b'(t)-c'(t)\|<\varepsilon$ and therefore $b'(t)\sim_h c'(t)$. By the continuity of $b'$ and $c'$, $b'(s)\sim_h c'(s')$ for all $s$ and $s'$, and the conclusion follows. 
\end{proof} 

By the Claim, we can define $\Theta\colon Z^C(\varphi)\to Z^B(\varphi)$ by 
\[
\Theta([\bar b]_h)=\Theta_0(\bar b).
\]
By considering constant paths in $B_\infty$, one sees that~$\Theta$ is surjective. 

Therefore, $\Psi\circ\Theta$ is the required equivariant surjection of $Z^C(\varphi)/\sim_h$ onto $Z^{A_\varphi}(\varphi)/\sim_h$. 

\eqref{3.L.standard} is an immediate consequence of \eqref{1.L.standard} and \eqref{2.L.standard}.
\end{proof}

\begin{proposition} \label{P.main} Suppose that a formula $\varphi$ admits definable homotopy in an axiomatizable class of \cstar-algebras $\cE$. Moreover assume that there is $\alpha\in \End^1(A_\varphi)$ such that for every $\bar b\in Z^{A_\varphi}(\varphi)$ the set 
\[
I_\alpha(\bar b):=\{n\in \bbN\mid \bar b\sim_h \alpha^n(\bar a)\text{ for some }\bar a\in Z^{A_\varphi}(\varphi)\}
\]
is finite. 
If $D$ is a unital abelian \cstar-algebra such that the action of $\End^1(A_\varphi)$ on $D$ is standard, then $D$ is not countably saturated. %If $\varphi$ is quantifier-free, then the type witnessing this can be chosen to be existential. \marginpar{IF: In order to prove that the type is existential we would need to work out the complexity of the `new quantifiers' (this was regrettably not done in \cite{Muenster}). Do we want to go there?}
\end{proposition}

\begin{proof} As before, let $X_\varphi$ denote the spectrum of $A_\varphi$, so that Lemma~\ref{L.Phi.2} implies that for an abelian \cstar-algebra $B$ we have $X_\varphi\supseteq \jsp(\bar a)$ if and only if $\bar a\in Z^B(\varphi)$. As in the proof of Proposition~\ref{P.action}, for every $\alpha\in \End^1(A_\varphi)$ there exists a function $\bar f_\alpha$ such that for all $\bar a$ in $A_\varphi$ we have 
\[
\alpha(\bar a)=\bar f_\alpha(\bar a). 
\]
Fix $\alpha\in \End^1(A)$ such that $I_\alpha(\bar b)$ is finite for all $\bar b$ in $A_\varphi$. For $m\geq 1$ write $\bar f_m:=\bar f_{\alpha^m}$, hence $f_{m,j}=f_{\alpha^m,j}$ for $j<k$. 

Since the zero-set of $\varphi$ is definable, applying quantification over it to definable predicates results in definable predicates (see \cite[Definition~3.2.3]{Muenster}). We will write $\inf_{\bar x, \varphi(\bar x)=0}$ and $\sup_{\bar x, \varphi(\bar x)=0}$ for the corresponding quantifiers. 

Consider the $k$-type $\bt(\bar x)$ whose conditions are $\varphi(\bar x)=0$, together with 
\begin{equation}\label{Eq.m}
\inf_{\bar y, \varphi(\bar y)=0} \|\bar f_{m}(\bar y) -\bar x\|=0
\end{equation}
for $m\geq 1$. 
We claim that every finite subset of $\bt$ is satisfied. Fix $\bar a\in Z^{D}(\varphi)$. For $n\geq 1$ let $\bar b$ in $D$ be defined by 
\[
\bar b:=\bar f_{n!}(\bar a)
\]
Note that $\varphi^B(\bar b)=0$. For $1\leq m \leq n$, with $\bar c:=\bar f_{n!/m}(\bar a)$, we have $\bar f_{m}(\bar c)=\bar f_{n!}(\bar a)=\bar b$, thus $\bar b$ satisfies all conditions as in \eqref{Eq.m} for $m\leq n$; since $n\geq 1$ was arbitrary, the type $\bt$ is finitely satisfiable. 

Assume, towards obtaining a contradiction, that $\bt$ is satisfied by some $\bar b$ in $D$. Then $\bar b\in Z^D(\varphi)$. Since the action is standard, there is a surjective equivariant $\Psi\colon Z^D(\varphi)/\sim_h\to Z^{A_\varphi}(\varphi)/\sim_h$. Fix $\bar c\in Z^{A_\varphi}(\varphi)$ such that $\Psi([\bar b]_h)=[\bar c]_h$. By the assumption, the set $I_\alpha(\bar c)$ is finite. Fix $m$ that does not belong to this set. 

Let $\e>0$ be as in the conclusion of Proposition~\ref{P.homotopy}. Since $\bar b$ satisfies $\bt$, there exists $\bar d\in Z^D(\varphi)$ such that $\|\bar f_m(\bar d)-\bar b\|<\e$. Therefore $\bar f_m(\bar d)\sim_h \bar b$, and $\alpha^m.\Psi([\bar d]_h)=\Psi([\bar b]_h)=[\bar c]_h$. This contradicts the choice of $m$, and completes the proof. 
\end{proof}

\begin{corollary} For $n\geq 2$, the corona algebra of $C_0(\bbR^n)$ is not countably saturated. 
\end{corollary}

\begin{proof} We claim that the corona algebra of $C_0(\bbR^n)$ is isomorphic to $C(S^{n-1})_\infty$ (see~\S\ref{Eq.Ainfty} for the definition of $A_\infty$). To see this, let $\bbD^n$ denote the $n$-dimensional open unit ball. Since the closure of $\bbD^n$ inside $\bbR^n$ is compact, we have 
\[
\cQ(C_0(\bbR^n))\cong \cQ(C_0(\bbR^n\setminus \bbD^n)).
\]
 Since every compact Hausdroff space $K$ satisfies 
\[
C(K)_\infty\cong \cQ(C_0(K\times [0,\infty))), 
\]
 it remains to see that $\bbR^n\setminus \bbD^n$ is homeomorphic to $S^{n-1}\times [0,\infty)$. The map $S^{n-1}\times [0,\infty)\mapsto \bbR^n\setminus \bbD^n$ that sends $(x_0,\dots, x_n,t)$ to $(tx_0, tx_1,\dots tx_n)$ is clearly a homeomorphism. 

To conclude the proof, we just have to compute the homotopy classes of maps $S^n\to S^n$. This is the $n$-th homotopy group of the sphere $S^n$, $\pi_n(S^n)$. It is well known (see e.g., \cite[Corollary 4.25]{Hatcher}) that $\pi_n(S^n)\cong\mathbb Z$. With such association, fix an endomorphism $\alpha$ of $C(S^n)$ that corresponds to the class of $1$. Then $\alpha$ has an infinite order. Recall that $\varphi_n$ was defined in Lemma~\ref{L.Sn}, and that it is controlled by $S^{n-1}$.  Every $\bar b\in Z^{C(S^n)}(\varphi_n)$ belongs to some homotopy class, and the set $I(\bar b)$ as in Proposition~\ref{P.main} consists of its divisors. It is therefore finite, and by Proposition~\ref{P.main} the conclusion follows. 
\end{proof}

\section{Reduced products}\label{S.redprod}

By \cite[Theorem~1.5]{FaSh:Rigidity} (also \cite[Theorem~16.5.1]{Fa:STCstar}), the reduced product $\prod A_n/\bigoplus A_n$ of \cstar-algebras (or other metric structures of the same language) is countably saturated. If all $A_n$ are unital, then this reduced product is isomorphic to the corona of $\bigoplus_n A_n$. 
Even if the $A_n$s are not unital, the quotient $\cM(\bigoplus A_n)/\bigoplus_n \cM(A_n)$ is also countably saturated. This is because we have $\mathcal M(\bigoplus A_n)\cong\prod\mathcal M(A_n)$ (see \cite[II.8.1.3]{Black:Operator})
and therefore $\cM(\bigoplus A_n)/\bigoplus_n \cM(A_n)\cong\prod\cM(A_n)/\bigoplus\cM(A_n)$, and its saturation again follows from \cite[Theorem~1.5]{FaSh:Rigidity}. 

\begin{corollary}
 There is a sequence of separable abelian \cstar-algebras $A_n$ such that the corona of $\bigoplus_n A_n$ is not countably saturated. 
\end{corollary}

\begin{proof} Suppose that $X_n$ is a locally compact connected noncompact space, and let $X=\bigsqcup X_n$ (the disjoint union of the spaces $X_n$). This space satisfies the hypotheses of Theorem~\ref{T1}, therefore the corona algebra of $C_0(X)$ is not quantifier-free saturated. Since $C_0(X)\cong\bigoplus C_0(X_n)$, by \cite[II.8.1.3]{Black:Operator}, $\mathcal M(C_0(X))$ is isomorphic to $\prod C_b(X_n)$, and therefore by Theorem~\ref{T1}, the algebra $\prod C_b(X_n)/\bigoplus C_0(X_n)$ is not quantifier-free countably saturated. 
\end{proof}

A \cstar-algebra is called \emph{projectionless} if it has no projections other than $0$ and (possibly) $1$.

\begin{theorem}\label{T.DirectSums} Let $A_n$ be a sequence of nonunital \cstar-algebras infinitely many of which are projectionless. Then the corona of $\bigoplus A_n$ is not quantifier-free saturated.

In particular, there is a sequence of simple separable \cstar-algebras $A_n$ such that the 
corona of $\bigoplus A_n$ is not quantifier-free saturated.
\end{theorem}

\begin{proof}
Let $A=\bigoplus A_n$. By \cite[II.8.1.3]{Black:Operator}, we can identify $\mathcal M(A)$ with $\prod \cM(A_n)$. Again, if $a\in\mathcal M(A)$ we denote by $\dot a$ its image in $\mathcal Q(A)$. 

We first prove the theorem in case all $A_n$ are projectionless. Let $a_n\in\mathcal M(A)$ be the sequence which is the identity of $\mathcal M(A_n)$ in the $n$-th entry and $0$ otherwise. Consider the type $\bt(x)$ given by $\dot x\geq 0$, $\norm{\dot x}=1$, $\dot x^2-\dot x=0$ and $\dot x\dot a_n=0$.

If $\bt_0\subseteq\bt$ is finite, fix $n$ such that $x\dot a_n=0$ does not appear in $\bt_0$. Then $\dot a_n$ satisfies $\bt_0$, so $\bt_0$ is consistent. 

We now show that $\bt$ is not realized. Suppose that $c=(c_n)\in\prod M(A_n)$ is such that $\dot c$ realizes $\bt$. Since $\dot c\dot a_n=0$, we have $c_n\in A_n$ for all $n$. In addition $\dot c^2-\dot c=0$, hence $c^2-c\in\bigoplus A_n$, and in particular $\norm{c_n^2-c_n}\to 0$ as $n\to\infty$. Furthermore, since $\norm{\dot c}=1$, we have that $c\notin\bigoplus A_n$, and in particular $\norm{c_n}\to 1$, so $\lim_n\norm{c_n}\neq 0$. Therefore, for $n$ large enough, by usual functional calculus, we can pick a nontrivial projection $d_n\in A_n$ such that $\norm{c_n-d_n}<\frac{1}{4}$. Since $A_n$ is projectionless, this is a contradiction.

In case not all the algebra $A_n$ are projectionless, let 
\[
X=\{n\mid A_n \text{ is projectionless}\}.
\]
Notice that $X$ is infinite. Denote by $p_X$ the sequence in $\prod \mathcal M(A_n)$ which is $1_{\mathcal M(A_n)}$ in case $n\in X$ and $0$ otherwise. Adding to the type the requirement that the realizing element is orthogonal to $1-\dot p_X$ gives a consistent not realized type.

For the second part, choose any sequence of projectionless, nonunital, simple \cstar-algebras from \cite{razak2002classification}. 
\end{proof}

The conclusion of Theorem~\ref{T.DirectSums} may fail if the technical condition that (infinitely many of) the algebras $A_n$ in are projectionless is dropped. For example, if each $A_n$ equals $c_0(\NN)$, then the corona of $\bigoplus A_n$ is naturally isomorphic to $\ell_\infty/c_0$, a countably saturated \cstar-algebra (see the proof of Corollary~\ref{Ex.notsigmacompact}).

\section{Saturation of non-abelian coronas}\label{S.Nonabel}

If $A$ is a \cstar-algebra then $Z(A)$ denotes its center. Since the center is the zero-set of the formula $\varphi(x)=\sup_{y}\norm{xy-yx}$, if $Z(A)$ is not countably saturated, neither is $A$. Furthermore, if the center is not quantifier-free saturated, then the failure of saturation of $A$ is witnessed by a universal type (this is optimal, see Proposition~\ref{P.QFsaturated} below). Suppose that~$A$ is a separable nonunital \cstar-algebra such that $Z(A)=C_0(X)$ for some locally compact space $X$. What can we say in general about $Z(\cQ(A))$? An ability to control $Z(\cQ(A))$ would, together with the results of this note, imply the failure of countable saturation of some corona algebras. A test case: fix a unital \cstar-algebra $B$ with trivial center, and let $X$ be a locally compact noncompact space. Consider $A=C_0(X,B)$. By \cite{akemann1973multipliers}, in this case $\cM(A)\cong C_b(X,B)$, and therefore $Z(\cM(A))\cong C(\beta X)$. By the main result of \cite{vesterstrom1971homomorphic}, we then have $\pi[Z(\cM(A))]=Z(\cQ(A))$, and consequently $Z(\cQ(A))\cong C(\beta X\setminus X)$, if and only if the maximal ideal space of $Z(\cM(A))$ separates that of $\cM(A)$. This reflects in a technical condition on the \cstar-algebra $B$, which we leave open for future investigations.

In the following `$\aleph_2$-saturated' means that every consistent type of cardinality $\aleph_1$ is realized. (According to this convention, `countable saturation' is called `$\aleph_1$-saturation' but we find the former terminology more appealing.) By the classical Hausdorff's gap construction, $\ell_\infty/c_0$ is not quantifier-free $\aleph_2$-saturated. Hausdorff's gap has been transferred to the Calkin algebra in \cite{Za-Av:Gaps} and to the corona of every \cstar-algebra with a sequential approximate unit consisting of projections in \cite[Theorem~14.2.1]{Fa:STCstar}. By a standard argument (e.g., \cite[Lemma~15.3.3]{Fa:STCstar}), this implies that the Calkin algebra is not degree-1 $\aleph_2$-saturated. Similar phaenomena (such as constructions of certain Luzin families) were analyzed for coronas of simple $\sigma$-unital \cstar-algebras in \cite{vaccaro2016obstructions}. Luciano Salvetti proved that the corona of  every separable,  nonunital, \cstar-algebra is not degree-1 $\aleph_2$-saturated, answering a question from the original version of this note.

If $A$ is finite-dimensional, then $A_\infty=C_b([0,\infty), A)/C_0([0,\infty), A)$ is countably saturated. The reason for this is that for every \cstar-algebra $B$, $A\otimes B$ is definable in the theory of  $B$ (the case when $A$ is a full matrix algebra is \cite[Appendix~C]{goldbring2014kirchberg}, also  \cite[Lemma~4.2.4]{Muenster}, and the general case follows easily since $A$ is a direct sum of finitely many full matrix algebras). This implies that every type over $A\otimes B$ can be rewritten as a type over $B$, and the new type is satisfiable (realized) if, and only if, the original one is. (See the discussion of  eq in \cite[\S 3]{Muenster}).) Thus countable saturation of $A_\infty$ follows by \cite[Theorem~1]{FaSh:Rigidity}. These algebras play an important role in the E-theory of Connes and Higson (\cite[\S 25]{blackadar1998k}) and in the Phillips--Weaver construction of an outer automorphism of the Calkin algebra (\cite{PhWe:Calkin}).

\begin{question} \label{Q.Ainfty} Is there an infinite-dimensional \cstar-algebra $A$ such that $A_\infty$ is countably saturated? 
\end{question}

In case $A$ is unital, $A_\infty$ is the corona of a $\sigma$-unital \cstar-algebra, and therefore countably degree-1 saturated. By Corollary~\ref{C1}, an infinite-dimensional \cstar-algebra $A$ such that $A_\infty$ is countably saturated cannot be abelian.

We suspect that the following question has a negative answer. 

\begin{question}\label{Q.simple}
Is there a separable, simple, and nonunital \cstar-algebra $A$ such that $\cQ(A)$ is countably saturated, or at least countably homogeneous? What about the stabilization of the Cuntz algebra $\cO_2$? 
\end{question}

We conclude the paper with an observation loosely related to Corollary~\ref{C00} which answers \cite[Question~6.5]{eagle2015saturation} and a question. 

\begin{proposition} \label{P.QFsaturated} There exists a \cstar-algebra that is countably quantifier-free saturated, but not countably saturated. 
\end{proposition}

\begin{proof} Let $A$ be the CAR algebra (or any other separable, simple, unital \cstar-algebra) and let $D$ be the \cstar-algebra of continuous functions on the Cantor space (or any other infinite-dimensional, unital, abelian subalgebra of $A$). Fix a nonprincipal ultrafilter $\cU$ on $\bbN$. Then the norm ultrapower of~$A$, $A_\cU$, is countably saturated (e.g., \cite[Theorem~16.4.1]{Fa:STCstar}). Identify $D$ with its diagonal image in $A_\cU$. By \cite[Corollary~2]{farah2016new}, the bicommutant of $D$ in $A_\cU$ is equal to~$D$. Therefore $Z(A_\cU\cap D')$ is equal to $(A_\cU\cap D')\cap (A_\cU\cap D')'=D$. Since $D$ is separable and infinite-dimensional, the type of a central element of $A_\cU$ at distance 1 from a fixed countable dense subset of $D$ is countable and approximately satisfiable, but not realized, in $A_\cU\cap D'$. Therefore the \cstar-algebra $A_\cU\cap D'$ is not countably saturated. However, it is easily seen to be countably quantifier-free saturated (see e.g., \cite[Corollary~16.5.3]{Fa:STCstar}). 
\end{proof}

\begin{question}
Is there an abelian \cstar-algebra that is countably quantifier-free saturated, but not countably saturated?
\end{question}

Notably, the only unital abelian \cstar-algebras that admit elimination of quantifiers are $\bbC$, $\bbC^2$, and $C(\text{Cantor space})$ (\cite{eagle2015quantifier}). 

\bibliographystyle{plain} %{amsalpha}
\bibliography{ifmainbib}
\end{document}